\documentclass{birkjour}
\usepackage{amssymb}
\usepackage{stmaryrd}
\usepackage{mathtools}
\newtheorem{thm}{Theorem}[section]
\newtheorem{cor}[thm]{Corollary}
\newtheorem{lem}[thm]{Lemma}

\theoremstyle{definition}
\newtheorem{defn}[thm]{Definition}
\theoremstyle{remark}
\newtheorem{rem}[thm]{Remark}
\newtheorem{ex}[thm]{Example}
\numberwithin{equation}{section}

\begin{document}
	
	%-------------------------------------------------------------------------
	% editorial commands: to be inserted by the editorial office
	%
	%\firstpage{1} \volume{228} \Copyrightyear{2004} \DOI{003-0001}
	%\seriesextraline{This is the Concrete Title of this Book\br H.E. R and S.T.C. W, Eds.}
	%\firstpage{1}

\title[\(\mathcal F\)-Transitivity on Metric Trees]
{\(\mathcal F\)-Transitivity of Translation Semigroups on Directed Metric Trees}

	%----------Author 1
	\author[Chen]{Xiang Chen}
	
	\address{School of Mathematics\\
		Tianjin University\\
		Tianjin 300350\\
		P.R. China}
	
	\email{2020233018@tju.edu.cn}
	
	\thanks{Corresponding author: Xiang Chen. The work was supported in part by the National Natural Science Foundation of China (Grant Nos. 12571088) and the Natural Science Foundation of Henan (No. 262300421862).}
	
	%----------Author 2
	\author[Zhang]{Li Zhang}
	
	\address{School of Mathematics and Statistics\\
		Nanyang Normal University\\
		Nanyang 473061\\
		P.R. China}
	
	\email{zhangli0977@126.com}
	
	%----------Author 3
	\author[Zhou]{Zehua Zhou}
	
	\address{School of Mathematics\\
		Tianjin University\\
		Tianjin 300350\\
		P.R. China}
	
	\email{zehuazhoumath@aliyun.com; zhzhou@tju.edu.cn}
	
	\subjclass{Primary 47A16; Secondary 47D06}
	
	\keywords{Directed metric tree, translation semigroup, Furstenberg family, \(\mathcal F\)-transitivity}
	
	\date{}
	%\dedicatory{aaaaaaa}
	%%% ----------------------------------------------------------------------

\begin{abstract}
	We study the \(\mathcal F\)-transitivity and topological
	\(\mathcal F\)-recurrence of left translation semigroups on weighted
	\(L^p\)-spaces over directed metric trees. Motivated by the recent work
	of Mangino and Vargas-Moreno on hypercyclicity and weak mixing for these
	semigroups, we investigate the different problem of
	\(\mathcal F\)-transitivity, where the entire return-time set is required
	to belong to a prescribed finitely invariant Furstenberg family. Assuming
	that the weight is \(p\)-admissible, we obtain necessary and sufficient
	integral conditions for both rooted and rootless trees, and establish the
	equivalence between \(\mathcal F\)-transitivity and topological
	\(\mathcal F\)-recurrence. In the rooted case, the criteria depend only on
	the weights along descendant edges. In the rootless case, an additional
	ancestor term and a measurable cancellation construction reflect the
	backward geometry of the tree. Our approach is based on measurable
	weighted minimization and quantitative estimates over subsets of the edge
	parameter interval. Examples on homogeneous rooted trees and rootless
	trees with a free left end show that branching may determine the dynamics,
	but forward branching alone does not guarantee
	\(\mathcal F\)-transitivity in the rootless setting.
\end{abstract}

%%% ----------------------------------------------------------------------
\maketitle

\section{Introduction}

A central problem in linear dynamics is to understand how linear operators
can exhibit complicated orbit behavior, including hypercyclicity, weak
mixing, and mixing; see \cite{CFS,CC}. Furstenberg families provide a unified
return-time framework for several of these properties. This framework has been used to study strong transitivity properties of
operators \cite{KR2}. It has also been applied to disjoint
\(\mathcal F\)-transitive and \(\mathcal F\)-transitive families of
composition operators on \(L^p\)-spaces \cite{Cx3}. More recently,
disjoint \(\mathcal F\)-transitivity and topological multiple recurrence
of weighted composition operators on \(H(\mathbb D)\) were investigated
in \cite{CRS1}.

The dynamics of strongly continuous semigroups forms another important part
of the subject. Frequently hypercyclic semigroups were studied in
\cite{EPA}, while specification properties for \(C_0\)-semigroups were
considered in \cite{SBF}. Among the basic examples are translation
semigroups. Their supercyclic behavior on complex sectors was investigated
in \cite{YXZ}, and their \(\mathcal F\)-transitivity and recurrence on such
domains were characterized in \cite{HXM}. Translation semigroups on locally
compact groups were studied in \cite{AAN}. These results indicate that the
dynamics of a translation semigroup depends strongly on both the geometry
of the underlying space and the behavior of the weight.

Directed trees provide a natural framework in which branching becomes part
of the dynamics. Weighted shifts on directed trees were systematically
studied in \cite{KGD}. Further dynamical properties, including \(\mathcal F\)-transitivity,
were investigated in \cite{CRR}, while chaotic weighted shifts were
considered in \cite{KDP}. Related operators on function spaces over trees
have also attracted attention. Forward and backward shifts on Hardy spaces
were studied in \cite{ARMA}, hypercyclic weighted backward shifts on the
little Hardy and little Lipschitz spaces in \cite{Ka3}, and composition
operators on the little Lipschitz space of a  tree in \cite{CMA}.

A directed metric tree is obtained by identifying each edge of a directed
tree with a copy of the unit interval. It therefore combines the discrete
branching structure of a tree with continuous translation along its edges.
Recently, Mangino and Vargas-Moreno \cite{MV} introduced left translation
semigroups on weighted \(L^p\)-spaces over directed metric trees. They
characterized the strong continuity of these semigroups and studied their
hypercyclicity and weak mixing in both the rooted and rootless settings.
Their dynamical conditions are formulated in terms of pointwise extremal
behavior of the weights along descendant and ancestor branches.

The problem considered here is different in scope. While \cite{MV}
characterizes hypercyclicity and weak mixing, \(\mathcal F\)-transitivity
requires the entire return-time set \(N(U,V)\) to belong to a fixed
finitely invariant Furstenberg family \(\mathcal F\). To control these
return-time sets, our criteria involve integrating the pointwise weighted
minimum over measurable subsets of the edge parameter. This integral
arises naturally from the \(L^p\)-norm and provides a quantitative
condition on each candidate return time. For
\(\mathcal F=\mathcal F_\infty\), the family of unbounded subsets of
\(\mathbb R_+\), our criterion gives an alternative characterization of
hypercyclicity and therefore agrees with \cite{MV} at the level of the
dynamical property. For other finitely invariant Furstenberg families,
the same framework captures finer return-time properties not treated in
\cite{MV}.

Our contributions are threefold. First, we establish a measurable integral
form of the weighted minimization formula for finite or countable families
of weights. Second, for every finitely invariant Furstenberg family
\(\mathcal F\), we characterize both \(\mathcal F\)-transitivity and
topological \(\mathcal F\)-recurrence of translation semigroups on rooted
directed metric trees, and prove that these two properties are equivalent.
Third, under the same assumption on \(\mathcal F\), we obtain analogous
characterizations for rootless directed metric trees, where an additional
ancestor term and a measurable cancellation construction are needed. For
appropriate standard choices of \(\mathcal F\), our criteria specialize to
characterizations of topological transitivity, weak mixing, and topological
mixing.

\section{Preliminaries}
\subsection{Directed metric trees}

In this subsection, following \cite{MV}, we recall the basic notation and
introduce the directed metric tree that will be used throughout the paper.

\begin{defn}
	A \emph{directed tree} is a directed graph \(G=(V,E)\) satisfying:
	\begin{enumerate}
		\item[\textnormal{(i)}] \(G\) is connected;
		\item[\textnormal{(ii)}] \(G\) contains no cycles;
		\item[\textnormal{(iii)}] every vertex \(v\in V\) has at most one parent;
		that is, there exists at most one \(u\in V\) such that \((u,v)\in E\);
		\item[\textnormal{(iv)}] The set \(V\) is countable.
	\end{enumerate}
\end{defn}
If \((u,v)\in E\), then \(u\) is called the \emph{parent} of \(v\), and \(v\)
a \emph{child} of \(u\). For the directed edge \(e=(u,v)\), the vertex \(u\)
is called the \emph{tail} of \(e\), while \(v\) is called its \emph{head}.
A vertex without children is called a \emph{leaf}, and a vertex without a
parent a \emph{root}. The tree is called \emph{rooted} if it has a root, and
\emph{rootless} otherwise.

The associated \emph{directed metric tree} \(L(G)\) is obtained by replacing
each edge \(e=(u,v)\in E\) with a copy of the unit interval \([0,1)\), whose
coordinate increases from the tail \(u\) to the head \(v\); thus, \(u\) and
\(v\) correspond to \(0\) and \(1\), respectively.

Following the notation of \cite[Section~1.3]{MV}, write
\[
E=\{e_i:i\in I\},
\]
where \(I\) is countable. For \(i\in I\), define \(M_0(i)=\{i\}\) and  for $n\in\mathbb N_0$,
\[
M_{n+1}(i)
=
\left\{
j\in I:
\text{the head of }e_k\text{ is the tail of }e_j
\text{ for some }k\in M_n(i)
\right\}.
\]
Equivalently, \(j\in M_n(i)\) if and only if there exists a sequence
\[
e_i=e_{i_0},e_{i_1},\ldots,e_{i_n}=e_j
\]
such that the head of \(e_{i_{r-1}}\) is the tail of \(e_{i_r}\) for
\(r=1,\ldots,n\). Moreover,
\[
M_n(i)\cap M_n(k)=\varnothing
\qquad
(i\neq k,\ n\in\mathbb N_0).
\]

A function \(f\) on \(L(G)\) will be identified with the family $
f=(f_i)_{i\in I}$,
where \(f_i:[0,1)\to\mathbb K\) denotes the restriction of \(f\) to the edge
\(e_i\), and \(\mathbb K=\mathbb R\) or \(\mathbb C\).

Let \(\rho=(\rho_i)_{i\in I}\) be a family of weights such that
\[
\rho_i\in L^1([0,1)),
\qquad
\rho_i(s)>0 \quad \text{for every }s\in[0,1).
\]
For \(1\leq p<\infty\), define
\[
L^p_{\rho_i}[0,1)
=
\left\{
h:[0,1)\to\mathbb K:
h \text{ is measurable and }
\int_0^1 |h(s)|^p\rho_i(s)\,ds<\infty
\right\},
\]
with norm
\[
\|h\|_{p,\rho_i}
=
\left(
\int_0^1 |h(s)|^p\rho_i(s)\,ds
\right)^{1/p}.
\]
The weighted \(L^p\)-space on the directed metric tree is defined by
\[
L^p_\rho(L(G))
=
\left\{
f=(f_i)_{i\in I}:
f_i\in L^p_{\rho_i}[0,1)\ \text{for every }i\in I,\ 
\sum_{i\in I}\|f_i\|_{p,\rho_i}^p<\infty
\right\}.
\]
It is a Banach space with norm
\[
\|f\|_{p,\rho}
=
\biggl(
\sum_{i\in I}\|f_i\|_{p,\rho_i}^p
\biggl)^{1/p}.
\]
Moreover, since \(I\) is countable and each \(L^p_{\rho_i}[0,1)\) is separable
for \(1\leq p<\infty\), the space \(L^p_\rho(L(G))\), being the
\(\ell^p\)-sum of the spaces \(L^p_{\rho_i}[0,1)\), is separable.

We shall also use the finite-edge-supported subspace
\[
\bigoplus_{i\in I}L^p_{\rho_i}[0,1)
=
\left\{
f=(f_i)_{i\in I}\in L^p_\rho(L(G)):
\operatorname{supp}f:=\{i\in I:f_i\neq0\}
\text{ is finite}
\right\}.
\]
Since \(C_c(0,1)\) is dense in \(L^p_{\rho_i}[0,1)\) for every
\(i\in I\), the algebraic direct sum
\(
\bigoplus_{i\in I}C_c(0,1)
\)
is dense in \(L^p_\rho(L(G))\).
\subsection{Left translation semigroups on directed metric trees}
\begin{defn}
	Let \(X\) be a Banach space. A one-parameter family
	\((T_t)_{t\geq0}\) is called a
	\emph{strongly continuous semigroup}, or a \(C_0\)-semigroup, on \(X\) if the
	following conditions hold:
	\begin{enumerate}
		\renewcommand{\labelenumi}{\textnormal{(\roman{enumi})}}
		\item \(T_0=I\);
		\item \(T_{t+s}=T_tT_s\) for all \(s,t\geq0\);
		\item $
		\lim_{s\to t}\|T_sx-T_tx\|=0 $
		for every \(x\in X\) and every \(t\geq0\).
	\end{enumerate}
\end{defn}
Following \cite{MV}, for \(t\geq0\) and \(s\in[0,1)\), set
\[
n(t,s)=\lfloor t+s\rfloor,
\qquad
\sigma(t,s)=t+s-n(t,s)\in[0,1),
\]
where \(\lfloor x\rfloor\) denotes the greatest integer less than or
equal to \(x\).

For \(t\geq0\), the left translation  \(T_t\) is formally given by
\[
(T_tf)_i(s)
=
\sum_{j\in M_{n(t,s)}(i)}
f_j\bigl(\sigma(t,s)\bigr),
\qquad i\in I,
\]
for almost every \(s\in[0,1)\), where \(f=(f_i)_{i\in I} \in L^p_\rho(L(G))\) and an empty
sum is understood to be zero. Since directed edge-paths in a tree are
unique, this formula is equivalent to the matrix formulation in
\cite{MV}. 

We shall use the following sufficient condition, which is the implication
\((ii)\Rightarrow(i)\) in \cite[Proposition~2.2]{MV}, stated in our
notation.

\begin{lem}\label{lem2.2}
	Let \(1\leq p<\infty\), let \(\rho=(\rho_i)_{i\in I}\) be a weight
	on \(L(G)\), and let \(M\geq1\) and \(\omega\in\mathbb R\). Assume
	that, for every \(i\in I\), \(t\geq0\), and almost every
	\(s\in[0,1)\), one of the following conditions holds:
	\[
	\rho_i(s)
	\leq
	Me^{\omega t}
	\inf_{j\in M_{n(t,s)}(i)}
	\rho_j\bigl(\sigma(t,s)\bigr),
	\qquad p=1,
	\]
	or
	\[
	\biggl(
	\sum_{j\in M_{n(t,s)}(i)}
	\frac{1}
	{\rho_j\bigl(\sigma(t,s)\bigr)^{1/(p-1)}}
	\biggr)^{p-1}
	\leq
	M^pe^{p\omega t}\frac{1}{\rho_i(s)},
	\qquad 1<p<\infty.
	\]
	Then the above formulas define a strongly continuous semigroup
	\((T_t)_{t\geq0}\) of bounded linear operators on
	\(L_\rho^p(L(G))\), and
	\[
	\|T_t\|\leq Me^{\omega t},
	\qquad t\geq0.
	\]
\end{lem}

\begin{defn}\label{def:p-admissible}
	Following \cite[Definition~2.4]{MV}, a weight \(\rho\) is called
	\emph{\(p\)-admissible} if there exist \(M\geq1\) and
	\(\omega\in\mathbb R\) such that the corresponding inequality in
	Lemma~\ref{lem2.2} holds for every \(i\in I\), \(t\geq0\), and
	\(s\in[0,1)\).
\end{defn}

\begin{rem}\label{rem:p-admissible}
	By \cite[Remark~2.5]{MV}, every \(p\)-admissible weight is bounded away
	from zero on each edge; namely,
	\[
	\inf_{s\in[0,1)}\rho_i(s)>0,
	\qquad i\in I.
	\]
\end{rem}
\begin{defn}
	A nonempty collection \(\mathcal F\) of subsets of \(\mathbb R_+\) is called a
	\emph{Furstenberg family} on \(\mathbb R_+\) if it is hereditary upward, that is,
	whenever \(A\in\mathcal F\) and \(A\subset B\subset \mathbb R_+\), one has
	\(B\in\mathcal F\). 
\end{defn}

\begin{defn}
	Let \(X\) be a topological vector space and
	\((T_t)_{t\geq0}\) a linear semigroup on \(X\). For nonempty open sets
	\(U,V\subset X\), define
	\[
	N(U,V)=\{t\geq0:T_t(U)\cap V\neq\varnothing\}.
	\]
	Given a Furstenberg family \(\mathcal F\) on \(\mathbb R_+\), the semigroup
	\((T_t)_{t\geq0}\) is called \emph{\(\mathcal F\)-transitive} if
	\(
	N(U,V)\in\mathcal F
	\)
	for all nonempty open sets \(U,V\subset X\), and
topologically	\emph{ \(\mathcal F\)-recurrent} if
	\(
	N(U,U)\in\mathcal F
	\)
	for every nonempty open set \(U\subset X\).
\end{defn}

\begin{rem}\label{rem:classical-families}
	A Furstenberg family \(\mathcal F\) on \(\mathbb R_+\) is called
	\emph{proper} if
	\[
	\mathcal F\neq\varnothing
	\qquad\text{and}\qquad
	\varnothing\notin\mathcal F,
	\]
	and \emph{finite invariant} if
	\[
	A\in\mathcal F,\ R>0
	\quad\Longrightarrow\quad
	A\cap(R,\infty)\in\mathcal F.
	\]
	Throughout this paper, all Furstenberg families are assumed to have these
	properties. 
\end{rem}

\section{\(\mathcal F\)-Transitivity of Translation Semigroups}
We begin with the following measurable version of the weighted minimization
formula, whose proof is based on an adaptation of the argument in
\cite[Theorem~4.3]{KGD}. It will be used in the construction of measurable
perturbation functions.
\begin{lem}\label{lem3.1}
Let \(1\leq p<\infty\), let \(E\subset\mathbb R\) be a Lebesgue
measurable set, let \(J\) be a nonempty finite or countable set, and let
\((\omega_j)_{j\in J}\) be a family of measurable functions from \(E\)
into \((0,\infty)\).
	Let \(h\) be a measurable function on \(E\). Assume that there exists
	\(j_0\in J\) such that
	\[
	|h|^p\omega_{j_0}\in L^1(E).
	\]
	For \(s\in E\), set
\[
\Theta_p(s)
=
\begin{cases}
	\displaystyle
	\inf_{j\in J}\omega_j(s), & p=1,\\[2mm]
	\displaystyle
	\biggl(
	\sum_{j\in J}\omega_j(s)^{-1/(p-1)}
	\biggr)^{1-p}, & 1<p<\infty,
\end{cases}
\]
where \((+\infty)^{1-p}=0\) when \(1<p<\infty\).
 Then the function \(\Theta_p\) is measurable. Moreover, for every \(\delta>0\),
there exist measurable functions
	$
	v_j:E\to[0,\infty),
	\;\; j\in J$
	such that
	\[
	\sum_{j\in J}v_j(s)=1
	\quad\text{for every }s\in E,
	\]
	and
	\[
	\int_E |h(s)|^p
	\sum_{j\in J}|v_j(s)|^p\omega_j(s)\,ds
	\leq
	\int_E |h(s)|^p\Theta_p(s)\,ds+\delta .
	\]
\end{lem}
\begin{proof}
	The measurability of \(\Theta_p\) follows immediately from that of
	the functions \(\omega_j\), since countable infima and countable sums
	of nonnegative measurable functions are measurable.
	
	\medskip
	\noindent\textit{Case 1.}
	We first consider the case \(p=1\). Enumerate \(J\) as
	\[
	J=\{j_1,j_2,\ldots\},
	\]
	with \(j_1=j_0\); if \(J\) is finite, the enumeration terminates. For each
	admissible \(N\), set
	\[
	m_N(s)=\min_{1\leq k\leq N}\omega_{j_k}(s),
	\qquad s\in E.
	\]
	Then \(m_N\) is measurable. For \(1\leq k\leq N\), define
\[
A_{k,N}
=
\left\{
s\in E:
\omega_{j_k}(s)=m_N(s)
\ \text{and}\
\omega_{j_r}(s)>m_N(s)
\text{ for every }1\leq r<k
\right\}.
\]
	Each \(A_{k,N}\) is measurable, and the sets
	\(A_{1,N},\ldots,A_{N,N}\) form a measurable partition of \(E\).
	Define
	\[
	v_{j_k,N}=\chi_{A_{k,N}},
	\qquad 1\leq k\leq N,
	\]
	and set \(v_{j,N}=0\) for all remaining \(j\in J\). Then
	\[
	\sum_{j\in J}v_{j,N}(s)=1
	\quad\text{for every }s\in E,
	\]
	and
	\[
	\sum_{j\in J}v_{j,N}(s)\omega_j(s)=m_N(s)
	\quad\text{for every }s\in E.
	\]
	
	If \(J\) is finite, take \(N=|J|\). Then
	\(
	m_N(s)=\inf_{j\in J}\omega_j(s)=\Theta_1(s),
	\)
	and the desired conclusion follows with equality.
	
	Now assume that \(J\) is countably infinite. Then
	\[
	m_N(s)\downarrow \inf_{j\in J}\omega_j(s)=\Theta_1(s)
	\quad\text{as }N\to\infty.
	\]
	Moreover, since \(j_1=j_0\),
	\[
	0\leq |h(s)|m_N(s)
	\leq |h(s)|\omega_{j_0}(s),
	\]
	and the right-hand side belongs to \(L^1(E)\). Hence, by the dominated
	convergence theorem,
	\[
	\int_E |h(s)|m_N(s)\,ds
	\longrightarrow
	\int_E |h(s)|\Theta_1(s)\,ds.
	\]
	Thus, for the given \(\delta>0\), there exists \(N\in\mathbb N\) such that
	\[
	\int_E |h(s)|m_N(s)\,ds
	\leq
	\int_E |h(s)|\Theta_1(s)\,ds+\delta.
	\]
	Fix such an \(N\) and set \(v_j=v_{j,N}\). This proves the case \(p=1\).
	
	\medskip
\noindent\textit{Case 2.}	
Let \(1<p<\infty\). If \(J\) is finite, put
	\(
	S(s)=\sum_{j\in J}\omega_j(s)^{-1/(p-1)}
	\)
	and define
	\[
	v_j(s)
	=
	\frac{\omega_j(s)^{-1/(p-1)}}{S(s)},
	\qquad j\in J.
	\]
	Then each \(v_j\) is nonnegative and measurable,
	\[
	\sum_{j\in J}v_j(s)=1
	\quad\text{for every }s\in E,
	\]
	and
	\[
	\sum_{j\in J}|v_j(s)|^p\omega_j(s)
	=
	S(s)^{1-p}
	=
	\Theta_p(s).
	\]
	The desired conclusion therefore follows with equality.
	
	Suppose now that \(J\) is countably infinite. Choose an enumeration
	\[
	J=\{j_1,j_2,\ldots\}
	\]
	such that \(j_1=j_0\), and set
	\[
	S_N(s)=\sum_{m=1}^N\omega_{j_m}(s)^{-1/(p-1)}.
	\]
	Define
	\[
	v_{j_m,N}(s)
	=
	\begin{cases}
		\displaystyle
		\frac{\omega_{j_m}(s)^{-1/(p-1)}}{S_N(s)},
		&1\leq m\leq N,\\[3mm]
		0,&m>N.
	\end{cases}
	\]
	Then for every \(s\in E\), we obtain
	\[
	\sum_{j\in J}v_{j,N}(s)=1
	\quad
	\text{and}\quad
	\sum_{j\in J}|v_{j,N}(s)|^p\omega_j(s)
	=
	S_N(s)^{1-p}.
	\]
	
Let
\[
S(s)=\sum_{j\in J}\omega_j(s)^{-1/(p-1)}.
\]
Then \(S_N(s)\) increases to \(S(s)\), while \(S_N(s)^{1-p}\) decreases to
\(S(s)^{1-p}\) as \(N\) tends to infinity, with the convention
\(
(+\infty)^{1-p}=0.
\)  Moreover,
	\[
	0\leq |h(s)|^pS_N(s)^{1-p}
	\leq |h(s)|^p\omega_{j_0}(s).
	\]
	By the dominated convergence theorem,
	\[
	\int_E |h(s)|^pS_N(s)^{1-p}\,ds
	\longrightarrow
	\int_E |h(s)|^pS(s)^{1-p}\,ds.
	\]
	Therefore, for some \(N\in\mathbb N\),
	\[
	\int_E |h(s)|^pS_N(s)^{1-p}\,ds
	\leq
	\int_E |h(s)|^p 	\Theta_p(s)\,ds+\delta.
	\]
Fix such an \(N\) and set \(v_j=v_{j,N}\) for \(j\in J\). Then,
these functions satisfy the required normalization and integral estimate,
which completes the proof.
\end{proof}
The following lemma is a continuous-time analogue of
\cite[Proposition~4.5]{CRR}.

\begin{lem}\label{lem3.2}
	Let \(X\) be a Banach space and let \((T_t)_{t\geq0}\) be a strongly
	continuous semigroup on \(X\). Assume that
	\[
	X_0
	=
	\left\{
	x\in X:
	\lim_{t\to\infty}\|T_tx\|=0
	\right\}
	\]
	is dense in \(X\). If \((T_t)_{t\geq0}\) is topologically
	\(\mathcal F\)-recurrent, then it is \(\mathcal F\)-transitive.
\end{lem}

\begin{proof}
	Let \(U,V\subset X\) be nonempty open sets and choose \(v\in V\).
	Since
	\[
	U-v:=\{u-v:u\in U\}
	\]
	is nonempty and open and \(X_0\) is dense in \(X\), choose
	\(
	y\in(U-v)\cap X_0.
	\)
	Thus \(y+v\in U\). By the continuity of addition, there exist a nonempty
	open neighborhood \(V'\) of \(v\) and a zero-neighborhood \(W\) such that
	\[
	y+V'\subset U
	\qquad\text{and}\qquad
	W+V'\subset V.
	\]
	
	Since \(y\in X_0\), there exists \(R>0\) such that
	\[
	T_ty\in W,
	\qquad t>R.
	\]
	By the topological \(\mathcal F\)-recurrence of the semigroup and the
	finite invariance of \(\mathcal F\),
	\[
	N(V',V')\cap(R,\infty)\in\mathcal F.
	\]
	For every \(t\in N(V',V')\cap(R,\infty)\), there exists \(z\in V'\)
	such that \(T_tz\in V'\). Hence
	\[
	y+z\in U
	\]
	and
	\[
	T_t(y+z)=T_ty+T_tz\in W+V'\subset V.
	\]
	Therefore,
	\[
	N(V',V')\cap(R,\infty)\subset N(U,V).
	\]
	Since \(\mathcal F\) is upward hereditary, \(N(U,V)\in\mathcal F\).
	Thus \((T_t)_{t\geq0}\) is \(\mathcal F\)-transitive.
\end{proof}
	For \(i\in I\), \(t\geq0\), and \(s\in[0,1)\), define
	\[
	\Phi_i(t,s)
	=
	\biggl(
	\sum_{j\in M_{n(t,s)}(i)}
	\frac{1}{\rho_j(t+s-n(t,s))^{1/(p-1)}}
	\biggr)^{1-p},
	\]
	where the sum is understood in the extended sense and \((+\infty)^{1-p}=0\).
	Since \(G\) has no leaves, \(M_{n(t,s)}(i)\neq\varnothing\) for every
	\(i\in I\), \(t\geq0\), and \(s\in[0,1)\).
	
	Whenever \(t\geq0\) is fixed, we write
	\[
	t=n_0+\tau,\qquad n_0=\lfloor t\rfloor,\qquad 0\leq\tau<1,
	\]
	so that
	\begin{equation}\label{3.1}
		n(t,s)
		=
		\begin{cases}
			n_0, & 0\leq s<1-\tau,\\
			n_0+1, & 1-\tau\leq s<1.
		\end{cases}
	\end{equation}
	Throughout this section, \(m\) denotes the Lebesgue measure on
	\(\mathbb R\).
\begin{thm}\label{thm3.1}
	Let $G$ be a rooted directed tree without leaves. Let $1<p<\infty$ and assume that
	$\rho=(\rho_i)_{i\in I}$ is a $p$-admissible weight so that the left translation
	semigroup $(T_t)_{t\geq 0}$ is a strongly continuous semigroup on
	$L^p_\rho(L(G))$.
	For a finite nonempty set \(K\subset I\), \(\varepsilon>0\), and
	\(0<\alpha<1\), define
	\[
	\mathcal N(K,\varepsilon,\alpha)
	=
	\Biggl\{
	t\geq0:
	\inf_{\substack{E\subset[0,1)\\ m(E)>\alpha}}
	\int_E\sum_{i\in K}\Phi_i(t,s)\,ds
	<\varepsilon
	\Biggr\},
	\]
	where the infimum is taken over all Lebesgue measurable subsets
	\(E\) of \([0,1)\).
	Then the following assertions are equivalent:
\begin{enumerate}
	\renewcommand{\labelenumi}{\textnormal{(\roman{enumi})}}
	\item $(T_t)_{t\geq 0}$ is $\mathcal F$-transitive on $L^p_\rho(L(G))$;
	\item for every finite nonempty set $K\subset I$, every $\varepsilon>0$ and
	every $0<\alpha<1$, one has
	\[
	\mathcal{N}(K,\varepsilon,\alpha)\in\mathcal F;
	\]
	\item \((T_t)_{t\geq0}\) is topologically \(\mathcal F\)-recurrent on $L^p_\rho(L(G))$.
\end{enumerate}
\end{thm}
\begin{proof}
The implication \(\mathrm{(i)}\Rightarrow\mathrm{(iii)}\) follows immediately
from the definition of \(\mathcal F\)-transitivity. Conversely, since the
finite-edge-supported functions are dense in \(X=L^p_\rho(L(G))\) and are
eventually annihilated by the rooted left translation semigroup,
Lemma~\ref{lem3.2} gives \(\mathrm{(iii)}\Rightarrow\mathrm{(i)}\). Therefore, it remains to prove the equivalence between \(\mathrm{(i)}\) and
	\(\mathrm{(ii)}\).
	
	\medskip
	\noindent
	\(\mathrm{(i)}\Rightarrow\mathrm{(ii)}\). Fix a finite nonempty set
	\(K\subset I\), \(\varepsilon>0\), and \(0<\alpha<1\). For each \(i\in I\),
	let \(\xi_i\) denote the function that is identically equal to \(1\) on
	the edge \(e_i\) and equal to \(0\) on all other edges. Set
	\[
	y_K=\sum_{i\in K}\xi_i.
	\]
	Since \(K\) is finite and \(\rho_i\in L^1([0,1))\) for every \(i\in I\),
	we have \(y_K\in L_\rho^p(L(G))\).
	
	By the \(p\)-admissibility of \(\rho\), each \(\rho_i\) is bounded away from zero
	on \([0,1)\). Let
	\[
	c_K:=
	\min_{i\in K}
	\operatorname*{inf}_{s\in[0,1)}\rho_i(s)>0 .
	\]
Here and throughout, \(|K|:=\operatorname{card}(K)\) denotes the
cardinality of \(K\). Choose \(\delta>0\) so small that
\begin{equation}\label{3.0}
	\frac{2^p|K|\delta^p}{c_K}<1-\alpha
	\quad\text{and}\quad
	2^p\delta^p<\varepsilon.
\end{equation}
Let
\[
U=B(0,\delta),
\qquad
V=B(y_K,\delta),
\]
where \(B(x,r)\) denotes the open ball in \(L_\rho^p(L(G))\) centered at
\(x\) with radius \(r\). 
	By the \(\mathcal F\)-transitivity of \((T_t)_{t\geq0}\), we have
	\[
	N(U,V):=\{t\geq0:T_t(U)\cap V\neq\varnothing\}\in\mathcal F .
	\]
	
	We shall show that
	\[
	N(U,V)\subset \mathcal{N}(K,\varepsilon,\alpha).
	\]
	Fix \(t\in N(U,V)\). Then there exists \(f\in U\) such that \(T_tf\in V\). Thus
	\[
	\|f\|_{p,\rho}<\delta,
	\qquad
	\|T_tf-y_K\|_{p,\rho}<\delta .
	\]
	For each \(i\in K\), define
	\[
	E_{t,i}
	=
	\left\{
	s\in[0,1):
	\bigl|(T_tf)_i(s)-1\bigr|<\frac12
	\right\},
	\qquad
	E_t=\bigcap_{i\in K}E_{t,i}.
	\]
	If \(s\in E_{t,i}^c\), then
	\(
	\bigl|(T_tf)_i(s)-1\bigr|\geq \frac12 .
	\)
	Therefore
	\[
	\begin{aligned}
		\frac{c_K}{2^p}m(E_{t,i}^c)
		\leq
		\int_{E_{t,i}^c}
		\bigl|(T_tf)_i(s)-1\bigr|^p\rho_i(s)\,ds \leq
		\|T_tf-y_K\|_{p,\rho}^p
		<\delta^p .
	\end{aligned}
	\]
	Hence
	\[
	m(E_{t,i}^c)<\frac{2^p\delta^p}{c_K},
	\qquad i\in K.
	\]
Since each \(E_{t,i}\) is measurable, the set
\(
E_t
\)
is also measurable. It follows from \eqref{3.0} that
\[
m(E_t^c)
\leq
\sum_{i\in K}m(E_{t,i}^c)
<
\frac{2^p |K|\delta^p}{c_K}
<
1-\alpha.
\]
Consequently, \(m(E_t)>\alpha\).
	
	For $s\in E_t$ and $i\in K$, we have
	$
	|(T_t f)_i(s)|\geq \frac12$.
	By the Hölder estimate,
\[
\begin{aligned}
	&\biggl|
	\sum_{j\in M_{n(t,s)}(i)}
	f_j(t+s-n(t,s))
	\biggr|^p
	\Phi_i(t,s)  \\
	&\qquad\leq
	\sum_{j\in M_{n(t,s)}(i)}
	|f_j(t+s-n(t,s))|^p
	\rho_j(t+s-n(t,s)).
\end{aligned}
\]
Thus, for \(s\in E_t\),
\[
\Phi_i(t,s)
\leq
2^p
\sum_{j\in M_{n(t,s)}(i)}
|f_j(t+s-n(t,s))|^p
\rho_j(t+s-n(t,s)).
\]
Using the decomposition of \eqref{3.1}, we obtain
\[
\begin{aligned}
	&\int_{E_t}\sum_{i\in K}\Phi_i(t,s)\,ds \\
	&\leq
	2^p
	\sum_{i\in K}
	\int_{E_t\cap[0,1-\tau)}
	\sum_{j\in M_{n_0}(i)}
	|f_j(s+\tau)|^p\rho_j(s+\tau)\,ds  \\
	&\quad+
	2^p
	\sum_{i\in K}
	\int_{E_t\cap[1-\tau,1)}
	\sum_{j\in M_{n_0+1}(i)}
	|f_j(s+\tau-1)|^p\rho_j(s+\tau-1)\,ds .
\end{aligned}
\]
	
	By the changes of variables \(u=s+\tau\) in the first integral and
	\(u=s+\tau-1\) in the second integral, we get
	\[
	\begin{aligned}
		\int_{E_t}\sum_{i\in K}\Phi_i(t,s)\,ds
		&\leq
		2^p
		\sum_{i\in K}
		\sum_{j\in M_{n_0}(i)}
		\int_{\tau}^{1}
		|f_j(u)|^p\rho_j(u)\,du        \\
		&\quad+
		2^p
		\sum_{i\in K}
		\sum_{j\in M_{n_0+1}(i)}
		\int_{0}^{\tau}
		|f_j(u)|^p\rho_j(u)\,du .
	\end{aligned}
	\]
Since, for each fixed integer \(n\), the sets \(M_n(i)\), \(i\in K\),
are pairwise disjoint, we obtain
\[
\begin{aligned}
	\int_{E_t}\sum_{i\in K}\Phi_i(t,s)\,ds
	&\leq
	2^p\sum_{j\in I}\int_{\tau}^{1}
	|f_j(u)|^p\rho_j(u)\,du\\
	&\quad+
	2^p\sum_{j\in I}\int_{0}^{\tau}
	|f_j(u)|^p\rho_j(u)\,du\\
	&=
	2^p\sum_{j\in I}\int_0^1
	|f_j(u)|^p\rho_j(u)\,du\\
	&=
	2^p\|f\|_{p,\rho}^p.
\end{aligned}
\]
Consequently,
\[
\int_{E_t}\sum_{i\in K}\Phi_i(t,s)\,ds
<
2^p\delta^p
<
\varepsilon.
\]
	Therefore $t\in \mathcal N(K,\varepsilon,\alpha)$. Hence
	\(N(U,V)\subset \mathcal{N}(K,\varepsilon,\alpha).\)
	Since $N(U,V)\in\mathcal F$ and $\mathcal F$ is upward hereditary, we have $
	\mathcal{N}(K,\varepsilon,\alpha)\in\mathcal F$.
	
\medskip
\noindent
\(\mathrm{(ii)}\Rightarrow\mathrm{(i)}\). Let \(U,V\) be two nonempty open
subsets of \(L^p_\rho(L(G))\). Since \(
\bigoplus_{i\in I}C_c(0,1)
\) is dense in
\(L^p_\rho(L(G))\), we may choose finite-edge-supported functions
\(f,g\in \bigoplus_{i\in I}C_c(0,1)\) such that
\[
f\in U,\qquad g\in V,
\]
and \(g\neq0\). Put
$
K:=\operatorname{supp}g$.
Then \(K\) is finite and nonempty. Choose \(\eta>0\) such that
\[
B(f,\eta)\subset U,
\qquad
B(g,\eta)\subset V.
\]

Since \(g\) has finite support and each \(g_i\) is continuous, set
\[M_g=\max_{i\in K}\|g_i\|_\infty<\infty.
\]
By the absolute continuity of the integral, choose \(0<\alpha<1\) so close
to \(1\) that, for every measurable set \(E\subset[0,1)\) with
\(m(E)<1-\alpha\),
\begin{equation}\label{3.2}
	\sum_{i\in K}\int_E |g_i(s)|^p\rho_i(s)\,ds
	<
	\left(\frac{\eta}{2}\right)^p.
\end{equation}
Choose \(\varepsilon>0\) and \(\gamma>0\) such that
\begin{equation}\label{3.3}
	M_g^p\varepsilon
	<
	\left(\frac{\eta}{2}\right)^p,
	\qquad
	|K|M_g^p\gamma
	<
	\left(\frac{\eta}{2}\right)^p.
\end{equation}

By assumption,
\[
\mathcal N(K,\varepsilon,\alpha)\in\mathcal F .
\]
Since \(f\) has finite edge support and \(G\) is rooted, there exists \(R>0\)
such that
\[
T_tf=0,\qquad t> R .
\]
By the finite invariance of \(\mathcal F\),
$
A:=\mathcal N(K,\varepsilon,\alpha)\cap (R,\infty)\in\mathcal F$.

We claim that
\[
A\subset N(U,V).
\]
Let \(t\in A\), and write, as in \eqref{3.1},
\[
t=n_0+\tau,
\qquad
n_0=\lfloor t\rfloor,
\qquad
0\leq\tau<1.
\]
Since \(t\in\mathcal N(K,\varepsilon,\alpha)\), there exists a measurable
set \(E_t\subset[0,1)\) such that
\[
m(E_t)>\alpha
\qquad\text{and}\qquad
\int_{E_t}\sum_{i\in K}\Phi_i(t,s)\,ds<\varepsilon.
\]

Put
\[
E_t^0=E_t\cap[0,1-\tau),
\qquad
E_t^1=E_t\cap[1-\tau,1),
\]
and set
\[
k_\ell=n_0+\ell,
\qquad \ell=0,1.
\]
By \eqref{3.1},
\[
n(t,s)=k_\ell,
\qquad s\in E_t^\ell,\quad \ell=0,1.
\]

Fix \(i\in K\). Applying Lemma~\ref{lem3.1} on each \(E_t^\ell\), with
\(h\equiv1\) and weights
\[
\omega_j(s)=\rho_j(t+s-k_\ell),
\qquad j\in M_{k_\ell}(i),
\]
we obtain nonnegative measurable functions
\[
v_{i,j}^{(\ell)}:E_t^\ell\to[0,\infty),
\qquad j\in M_{k_\ell}(i),
\]
such that
\[
\sum_{j\in M_{k_\ell}(i)}v_{i,j}^{(\ell)}(s)=1
\quad\text{for every }s\in E_t^\ell,
\]
and
\[
\begin{aligned}
	&\int_{E_t^\ell}
	\sum_{j\in M_{k_\ell}(i)}
	|v_{i,j}^{(\ell)}(s)|^p
	\rho_j(t+s-k_\ell)\,ds\\
	&\qquad\leq
	\int_{E_t^\ell}
	\biggl(
	\sum_{j\in M_{k_\ell}(i)}
	\frac{1}{\rho_j(t+s-k_\ell)^{1/(p-1)}}
	\biggr)^{1-p}ds
	+\frac{\gamma}{2}.
\end{aligned}
\]
For \(s\in E_t\), define
\[
v_{i,j}(s):=
\begin{cases}
	v_{i,j}^{(0)}(s),
	& s\in E_t^0,\ j\in M_{k_0}(i),\\[1mm]
	v_{i,j}^{(1)}(s),
	& s\in E_t^1,\ j\in M_{k_1}(i),\\[1mm]
	0,
	& \text{otherwise}.
\end{cases}
\] Then
\[
\sum_{j\in M_{n(t,s)}(i)}v_{i,j}(s)=1
\quad\text{for every }s\in E_t,
\]
and
\begin{equation}\label{3.4}
	\int_{E_t}
	\sum_{j\in M_{n(t,s)}(i)}
	|v_{i,j}(s)|^p\rho_j(t+s-n(t,s))\,ds
	\leq
	\int_{E_t}\Phi_i(t,s)\,ds+\gamma.
\end{equation}

Define
\[
F_t^0:=\{s+\tau:s\in E_t^0\},
\qquad
F_t^1:=\{s+\tau-1:s\in E_t^1\}.
\]
For \(j\in I\) and \(w\in[0,1)\), define
\[
(u_t)_j(w):=
\begin{cases}
	g_i(w-\tau)v_{i,j}(w-\tau),
	&
	\begin{array}{l}
		w\in F_t^0,\quad j\in M_{n_0}(i)\\[-1mm]
		\text{for some }i\in K,
	\end{array}
	\\[3mm]
	g_i(w+1-\tau)v_{i,j}(w+1-\tau),
	&
	\begin{array}{l}
		w\in F_t^1,\quad j\in M_{n_0+1}(i)\\[-1mm]
		\text{for some }i\in K,
	\end{array}
	\\[3mm]
	0,
	& \;\;\text{otherwise}.
\end{cases}
\]
Since the sets \(M_n(i)\), \(i\in K\), are pairwise disjoint for each
fixed \(n\), and \(F_t^0\cap F_t^1=\varnothing\), the above definition is
unambiguous and measurable.

Since all the terms involved are nonnegative, Tonelli's theorem allows
us to interchange the relevant sums and integrals. Using \eqref{3.4}, we have
\[
\begin{aligned}
	\|u_t\|_{p,\rho}^p
	&=
	\int_{E_t}
	\sum_{i\in K}|g_i(s)|^p
	\sum_{j\in M_{n(t,s)}(i)}
	|v_{i,j}(s)|^p
	\rho_j(t+s-n(t,s))\,ds\\
	&\leq
	M_g^p
	\sum_{i\in K}
	\left(
	\int_{E_t}\Phi_i(t,s)\,ds+\gamma
	\right)\\
	&=
	M_g^p
	\int_{E_t}\sum_{i\in K}\Phi_i(t,s)\,ds
	+
	|K|M_g^p\gamma<
	M_g^p\varepsilon+|K|M_g^p\gamma .
\end{aligned}
\]
Then, by \eqref{3.3},
$
\|u_t\|_{p,\rho}^p<\eta^p$.
Hence
\(
f+u_t\in B(f,\eta)\subset U.
\)

Since \(t> R\), we have \(T_tf=0\). Hence
\(
T_t(f+u_t)=T_tu_t .
\)
By construction, for every \(i\in K\),
\[
\begin{aligned}
	(T_tu_t)_i(s)
	&=
	\sum_{j\in M_{n(t,s)}(i)}
	(u_t)_j(t+s-n(t,s))\\
	&=
	g_i(s)\chi_{E_t}(s)
	\sum_{j\in M_{n(t,s)}(i)}v_{i,j}(s)\\
	&=
	g_i(s)\chi_{E_t}(s)
\end{aligned}
\]
for a.e. \(s\in[0,1)\). Moreover, if \(i\notin K\), then the uniqueness of
ancestors in the tree implies that the above construction gives no contribution
to the \(i\)-th coordinate. Thus
\[
(T_tu_t)_i(s)=0,\qquad i\notin K,
\]
for a.e. \(s\in[0,1)\). Since \(\operatorname{supp}g\subset K\), it follows that
\[
\begin{aligned}
	\|T_t(f+u_t)-g\|_{p,\rho}^p
	&=
	\sum_{i\in K}
	\int_{E_t^c}
	|g_i(s)|^p\rho_i(s)\,ds .
\end{aligned}
\]
Since \(m(E_t^c)<1-\alpha\), it follows from
\eqref{3.2} that
\[
\|T_t(f+u_t)-g\|_{p,\rho}<\frac{\eta}{2}.
\]
Therefore
\[
T_t(f+u_t)\in B(g,\eta)\subset V.
\]
Thus \(t\in N(U,V)\). Hence
\(
A\subset N(U,V).
\)
Since \(A\in\mathcal F\) and \(\mathcal F\) is upward hereditary, we obtain
$
N(U,V)\in\mathcal F$.
Therefore \((T_t)_{t\geq0}\) is \(\mathcal F\)-transitive.
\end{proof}

\begin{cor}
	Let \(G\) be a rooted directed tree without leaves. Let
	\(\rho=(\rho_i)_{i\in I}\) be a \(1\)-admissible weight such that the left
	translation semigroup \((T_t)_{t\geq0}\) is a strongly continuous semigroup on
	\(L^1_\rho(L(G))\). 
	For a finite nonempty set \(K\subset I\), \(\varepsilon>0\), and
	\(0<\alpha<1\), define
	\[
	\mathcal N_1(K,\varepsilon,\alpha)
	=
	\Biggl\{
	t\geq0:
	\inf_{\substack{E\subset[0,1)\\ m(E)>\alpha}}
	\int_E
	\sum_{i\in K}
	\inf_{j\in M_{n(t,s)}(i)}
	\rho_j(t+s-n(t,s))
	\,ds
	<\varepsilon
	\Biggr\}.
	\]
	Then the following assertions are equivalent:
	\begin{enumerate}
		\renewcommand{\labelenumi}{\textnormal{(\roman{enumi})}}
		\item \((T_t)_{t\geq0}\) is \(\mathcal F\)-transitive on
		\(L^1_\rho(L(G))\);
		\item for every finite nonempty set \(K\subset I\), every
		\(\varepsilon>0\), and every \(0<\alpha<1\), one has
		\[
		\mathcal N_1(K,\varepsilon,\alpha)\in\mathcal F .
		\]
		\item \((T_t)_{t\geq0}\) is topologically \(\mathcal F\)-recurrent on $L^1_\rho(L(G))$.
	\end{enumerate}
\end{cor}
\begin{proof}
	The proof follows the same argument as that of
	Theorem~\ref{thm3.1}, using the case \(p=1\) of
	Lemma~\ref{lem3.1}. In the necessity part, the weighted Hölder
	estimate is replaced by
	\[
	\bigl|\sum_{j\in J}a_j\bigr|
	\inf_{j\in J}\omega_j
	\leq
	\sum_{j\in J}|a_j|\omega_j.
	\]
\end{proof}

Assume that \(G\) is rootless. For \(i\in I\) and \(n\in\mathbb N_0\),
let \(K_n(i)\) denote the unique \(n\)-th ancestor edge of \(e_i\); that is,
\[
K_0(i)=i
\qquad\text{and}\qquad
i\in M_n\bigl(K_n(i)\bigr).
\]
The existence and uniqueness of \(K_n(i)\) follow from the rootlessness of
\(G\) and the uniqueness of directed paths.

For simplicity, we shall use the following notation throughout this section. 
For \(1<p<\infty\), define
\[
\Psi_i(t,s)
:=
\biggl(
\rho_{K_{n(t,s)}(i)}(s)^{-1/(p-1)}
+
\sum_{j\in M_{n(t,s)}(K_{n(t,s)}(i))}
\rho_j(\sigma(t,s))^{-1/(p-1)}
\biggr)^{1-p}.
\]
\begin{thm}\label{thm3.5}
	Let \(G\) be a rootless directed tree without leaves.
	Let \(1<p<\infty\), and let \(\rho=(\rho_i)_{i\in I}\) be a \(p\)-admissible
	weight such that the left translation semigroup \((T_t)_{t\geq0}\) is strongly
	continuous on \(L^p_\rho(L(G))\).
	
	For every finite nonempty set \(K\subset I\), \(\varepsilon>0\), and
	\(0<\alpha<1\), define
	\[
	\mathcal N(K,\varepsilon,\alpha)
	=
	\Biggl\{
	t\geq0:
	\inf_{\substack{E\subset[0,1)\ \mathrm{measurable}\\ m(E)>\alpha}}
	\int_E
	\sum_{i\in K}\bigl(\Phi_i(t,s)+\Psi_i(t,s)\bigr)\,ds
	<\varepsilon
	\Biggr\}.
	\]
	Then the following assertions are equivalent:
	\begin{enumerate}
		\renewcommand{\labelenumi}{\textnormal{(\roman{enumi})}}
		\item \((T_t)_{t\geq0}\) is \(\mathcal F\)-transitive on \(L^p_\rho(L(G))\);
		\item for every finite nonempty set \(K\subset I\), every
		\(\varepsilon>0\), and every \(0<\alpha<1\),
		\[
		\mathcal N(K,\varepsilon,\alpha)\in\mathcal F,
		\]
			\item \((T_t)_{t\geq0}\) is topologically \(\mathcal F\)-recurrent on \(L^p_\rho(L(G))\).
	\end{enumerate}
\end{thm}

\begin{proof}
	It is clear that \(\mathrm{(i)}\Rightarrow\mathrm{(iii)}\).
	
	\medskip
	\noindent
	\(\mathrm{(iii)}\Rightarrow\mathrm{(ii)}\).
	Fix a finite nonempty set \(K\subset I\), \(\varepsilon>0\), and
	\(0<\alpha<1\). For each \(i\in I\), let \(\xi_i\) be the function which is
	identically equal to \(1\) on the edge \(e_i\) and equal to \(0\) on all other
	edges. Put
	\[
	y_K=\sum_{i\in K} \xi_i.
	\]
	Since \(K\) is finite, \(y_K\in L^p_\rho(L(G))\).
	
	By \(p\)-admissibility, each \(\rho_i\) is bounded away from \(0\)
	on \([0,1)\). Hence
	\[
	c_K=
	\min_{i\in K}\operatorname*{inf}_{s\in[0,1)}\rho_i(s)>0 .
	\]
	Choose \(\delta>0\) so small that
	\[
	\frac{2^p |K|\delta^p}{c_K}<1-\alpha
	\quad\text{and}\quad
	\bigl(2^p+2|K|\bigr)\delta^p<\varepsilon .
	\]
	
	Let
	\(
	U=B(y_K,\delta).
	\)
	Since \((T_t)_{t\geq0}\) is topologically \(\mathcal F\)-recurrent, we have
	\[
	N(U,U):=\{t\geq0:T_t(U)\cap U\neq\varnothing\}\in\mathcal F .
	\]
	
	Since \(K\) is finite and \(G\) is a directed tree, there exists an integer
	\(N\geq1\) such that, for every \(n\geq N\) and every \(i\in K\),
	\begin{equation}\label{3.5}
		M_n(i)\cap K=\varnothing,
		\qquad
		K_n(i)\notin K,
		\quad\ i\in K.
	\end{equation}

	Choose \(R>0\) so large that
	\[
	n(t,s)\geq N,\qquad t> R,\ s\in[0,1).
	\]
	By the finite invariance of \(\mathcal F\),
	\[
	A:=N(U,U)\cap (R,\infty)\in\mathcal F .
	\]
	We will prove that
	\[
	A\subset \mathcal N(K,\varepsilon,\alpha).
	\]
	
Let \(t\in A\). Then there exists \(f\in U\) such that \(T_tf\in U\).
Hence
\begin{equation}\label{3.6}
	\|f-y_K\|_{p,\rho}<\delta,
	\qquad
	\|T_tf-y_K\|_{p,\rho}<\delta.
\end{equation}
For each \(i\in K\), set
\[
E_{t,i}:=
\left\{
s\in[0,1):
\bigl|(T_tf)_i(s)-1\bigr|<\frac12
\right\},
\qquad
E_t:=\bigcap_{i\in K}E_{t,i}.
\]
As in the proof of Theorem~\ref{thm3.1}, using the definition of \(c_K\)
and \eqref{3.6}, we obtain
\[
m(E_{t,i}^c)<\frac{2^p\delta^p}{c_K},
\qquad i\in K.
\]
Consequently, we obtain
\[
m(E_t^c)
\leq
\sum_{i\in K}m(E_{t,i}^c)
<
\frac{2^p|K|\delta^p}{c_K}
<
1-\alpha,
\]
and hence \(m(E_t)>\alpha\).
	
We first estimate the descendant term. For \(s\in E_t\) and \(i\in K\), put
\[
n=n(t,s),\qquad \sigma=\sigma(t,s).
\]
Since
\[
\bigl|(T_tf)_i(s)\bigr|\geq \frac12\quad
\text{and}\quad
(T_tf)_i(s)=\sum_{j\in M_n(i)}f_j(\sigma),
\]
the weighted Hölder estimate gives
\[
\Phi_i(t,s)
\leq
2^p
\sum_{j\in M_n(i)}
|f_j(\sigma)|^p\rho_j(\sigma).
\]
Using the decomposition of \([0,1)\) introduced above, the changes of
variables \(u=s+\tau\) and \(u=s+\tau-1\), and
\eqref{3.5}, we obtain
\[
\begin{aligned}
	&\int_{E_t}\sum_{i\in K}\Phi_i(t,s)\,ds
	\leq
	2^p
	\sum_{i\in K}
	\int_{E_t\cap[0,1-\tau)}
	\sum_{j\in M_{n_0}(i)}
	|f_j(s+\tau)|^p\rho_j(s+\tau)\,ds\\
	&\quad+
	2^p
	\sum_{i\in K}
	\int_{E_t\cap[1-\tau,1)}
	\sum_{j\in M_{n_0+1}(i)}
	|f_j(s+\tau-1)|^p\rho_j(s+\tau-1)\,ds\\
	&\leq
	2^p
	\sum_{j\in I}
	\int_{\tau}^{1}
	\bigl|(f-y_K)_j(u)\bigr|^p\rho_j(u)\,du+
	2^p
	\sum_{j\in I}
	\int_{0}^{\tau}
	\bigl|(f-y_K)_j(u)\bigr|^p\rho_j(u)\,du\\
	&=
	2^p\|f-y_K\|_{p,\rho}^p
	<
	2^p\delta^p .
\end{aligned}
\]

Next we estimate the ancestral-branch term. Fix \(s\in E_t\) and \(i\in K\).
Since \(n(t,s)\geq N\), it follows from
\eqref{3.5} that \(K_n(i)\notin K\). Hence
\[
(y_K)_{K_n(i)}(s)=0.
\]
Moreover, since \(i\in M_n(K_n(i))\), the set
\(M_n(K_n(i))\cap K\) is nonempty. Let
\[
q=\bigl|M_n(K_n(i))\cap K\bigr|.
\]
Then \(1\leq q\leq |K|\) and
\[
\sum_{j\in M_n(K_n(i))}(y_K)_j(\sigma)=q .
\]
Then
we get
\[
(T_tf)_{K_n(i)}(s)
-
\sum_{j\in M_n(K_n(i))}
\bigl(f_j(\sigma)-(y_K)_j(\sigma)\bigr)
=q .
\]
Dividing by \(q\) and applying the weighted Hölder estimate with weights
\(\rho_{K_n(i)}(s)\) and \(\rho_j(\sigma)\),
\(j\in M_n(K_n(i))\), we obtain
\[
\begin{aligned}
	\Psi_i(t,s)
	&\leq
	\frac{1}{q^p}
	\biggl(
	\bigl|(T_tf)_{K_n(i)}(s)\bigr|^p\rho_{K_n(i)}(s)
	+
	\sum_{j\in M_n(K_n(i))}
	\bigl|f_j(\sigma)-(y_K)_j(\sigma)\bigr|^p\rho_j(\sigma)
	\biggl)\\
	&\leq
	\bigl|(T_tf)_{K_n(i)}(s)\bigr|^p\rho_{K_n(i)}(s)
	+
	\sum_{j\in M_n(K_n(i))}
	\bigl|f_j(\sigma)-(y_K)_j(\sigma)\bigr|^p\rho_j(\sigma).
\end{aligned}
\]

Integrating over \(E_t\) and summing over \(i\in K\), we get
\[
\begin{aligned}
	\int_{E_t}\sum_{i\in K}\Psi_i(t,s)\,ds
	&\leq
	\int_{E_t}
	\sum_{i\in K}
	\bigl|(T_tf)_{K_n(i)}(s)\bigr|^p
	\rho_{K_n(i)}(s)\,ds\\
	&\quad+
	\int_{E_t}
	\sum_{i\in K}
	\sum_{j\in M_n(K_n(i))}
	\bigl|f_j(\sigma)-(y_K)_j(\sigma)\bigr|^p
	\rho_j(\sigma)\,ds,
\end{aligned}
\]
where \(n=n(t,s)\) and \(\sigma=\sigma(t,s)\). Moreover, for each fixed \(s\in E_t\), every coordinate is counted at
most \(|K|\) times. Since all the terms are nonnegative, Tonelli's
theorem shows that the first term is bounded by
\[
|K|\|T_tf-y_K\|_{p,\rho}^p.
\]
Indeed, \(K_n(i)\notin K\), so
\[
(T_tf)_{K_n(i)}
=
(T_tf-y_K)_{K_n(i)},
\]
and at most \(|K|\) such terms are counted.

For the second term, using again the decomposition of \([0,1)\) and the
changes of variables \(u=s+\tau\) and \(u=s+\tau-1\), each edge is counted at
most \(|K|\) times. Hence the second term is bounded by
\[
|K|\|f-y_K\|_{p,\rho}^p .
\]
Consequently, by \eqref{3.6},
\[
\begin{aligned}
	\int_{E_t}\sum_{i\in K}\Psi_i(t,s)\,ds
	\leq
	|K|\|T_tf-y_K\|_{p,\rho}^p
	+
	|K|\|f-y_K\|_{p,\rho}^p
	<
	2|K|\delta^p .
\end{aligned}
\]

Combining the estimates for \(\Phi_i\) and \(\Psi_i\), we get
\[
\int_{E_t}
\sum_{i\in K}\bigl(\Phi_i(t,s)+\Psi_i(t,s)\bigr)\,ds
<
\bigl(2^p+2|K|\bigr)\delta^p
<
\varepsilon .
\]
Since \(m(E_t)>\alpha\), this implies
\(
t\in\mathcal N(K,\varepsilon,\alpha).
\)
Thus
\(
A\subset \mathcal N(K,\varepsilon,\alpha).
\)
Since \(A\in\mathcal F\) and \(\mathcal F\) is upward hereditary, it follows that
\(
\mathcal N(K,\varepsilon,\alpha)\in\mathcal F.
\)
This proves \(\mathrm{(ii)}\).
	
\noindent
\(\mathrm{(ii)}\Rightarrow\mathrm{(i)}\).
We divide the proof into several steps.

\medskip
\noindent\textit{Step 1. Preliminary choices.}
Let \(U,V\) be two nonempty open subsets of \(L^p_\rho(L(G))\). Since \(\bigoplus_{i\in I}C_c(0,1)\) is dense in
\(L^p_\rho(L(G))\), choose nonzero functions
\[
f\in U\cap\bigoplus_{i\in I}C_c(0,1),
\qquad
g\in V\cap\bigoplus_{i\in I}C_c(0,1),
\]
and \(\eta>0\) such that
\[
B(f,\eta)\subset U,
\qquad
B(g,\eta)\subset V.
\]
Set
\(
K:=\operatorname{supp}f\cup\operatorname{supp}g.
\)
Then \(K\subset I\) is finite and nonempty.
\[
M_{f,g}
=
\max\left\{
1,\max_{i\in K}\bigl(\|f_i\|_\infty+\|g_i\|_\infty\bigr)
\right\}.
\]

By the absolute continuity of the integral, choose \(0<\alpha<1\) so close to
\(1\) that, whenever \(B\subset[0,1)\) is measurable and \(m(B)<1-\alpha\),
\begin{equation}\label{3.7}
	\sum_{i\in K}\int_B |f_i(s)|^p\rho_i(s)\,ds
	<
	\frac{\eta^p}{4},
	\qquad
	\sum_{i\in K}\int_B |g_i(s)|^p\rho_i(s)\,ds
	<
	\frac{\eta^p}{4}
\end{equation}
Next choose \(\varepsilon>0\) and \(\gamma>0\) so small that
\begin{equation}\label{3.8}
	|K|^{p-1}M_{f,g}^p\varepsilon<\frac{\eta^p}{8},
	\qquad
	\gamma<\frac{\eta^p}{8}.
\end{equation}
By assumption,
\(
\mathcal N(K,\varepsilon,\alpha)\in\mathcal F.
\)
Choose \(R>0\) so large that, for every \(t>R\), \(s\in[0,1)\), and
\(i\in K\),
\[
M_{n(t,s)}(i)\cap K=\varnothing,
\qquad
K_{n(t,s)}(i)\notin K.
\]
Set
\[
A:=\mathcal N(K,\varepsilon,\alpha)\cap(R,\infty).
\]
By \(\mathrm{(ii)}\) and the finite invariance of \(\mathcal F\), we have
\(A\in\mathcal F\). We shall prove that
\(
A\subset N(U,V).
\)

Fix \(t\in A\). Then there exists a measurable set \(E_t\subset[0,1)\) such
that
\[
m(E_t)>\alpha
\qquad\text{and}\qquad
\int_{E_t}
\sum_{i\in K}\bigl(\Phi_i(t,s)+\Psi_i(t,s)\bigr)\,ds
<\varepsilon.
\]

\medskip
\noindent\textit{Step 2. Cancellation on \(E_t^c\).}
Write
\[
t=n_0+\tau,
\qquad
n_0=\lfloor t\rfloor,
\qquad
0\leq\tau<1.
\]
Let \(E_t^c=[0,1)\setminus E_t\), and set
\[
B_t=
\bigl(E_t^c\cap[0,1-\tau)\bigr)+\tau
\ \cup\
\bigl(E_t^c\cap[1-\tau,1)\bigr)+\tau-1.
\]
Define
\[
(w_t^{(0)})_i=-f_i\chi_{B_t},
\qquad i\in K,
\]
and set \(w_t^{(0)}=0\) on all other edges. The two sets in the above
union are disjoint, and translations preserve Lebesgue measure. Hence
\[
m(B_t)=m(E_t^c)<1-\alpha.
\]
It follows from \eqref{3.7} that
\begin{equation}\label{3.9}
	\|w_t^{(0)}\|_{p,\rho}^p
	=
	\sum_{i\in K}\int_{B_t}|f_i(s)|^p\rho_i(s)\,ds
	<
	\frac{\eta^p}{4}.
\end{equation}

\medskip
\noindent\textit{Step 3. Cancellation of \(f\) on \(E_t\).}
As in the proof of Theorem~\ref{thm3.1}, let \(k_\ell\) be defined by
\(n(t,s)=k_\ell\) for \(s\in E_t^\ell\), \(\ell=0,1\). For \(\ell=0,1\), set
\[
\mathcal A_{t,\ell}
:=
\{K_{k_\ell}(i):i\in\operatorname{supp}f\}.
\]
For \(a\in\mathcal A_{t,\ell}\), define
\[
I_{a,\ell}
:=
\{i\in\operatorname{supp}f:K_{k_\ell}(i)=a\},
\qquad
F_{a,\ell}(s)
:=
\sum_{i\in I_{a,\ell}}f_i(s+\tau-\ell),
\quad s\in E_t^\ell.
\]
Apply Lemma~\ref{lem3.1} on \(E_t^\ell\), with
\[
h=F_{a,\ell},
\qquad
\omega_a(s)=\rho_a(s),
\qquad
\omega_j(s)=\rho_j(s+\tau-\ell),
\quad j\in M_{k_\ell}(a),
\]
and with error \(\gamma/(4|K|)\). This is possible since \(F_{a,\ell}\) is bounded and
\(\rho_a\in L^1((0,1))\). Consequently,
\[
\int_{E_t^\ell}|F_{a,\ell}(s)|^p\rho_a(s)\,ds<\infty.
\]
Thus, there exist nonnegative measurable functions
\[
\alpha_a^{(\ell)}
\quad\text{and}\quad
\beta_{a,j}^{(\ell)},
\qquad j\in M_{k_\ell}(a),
\]
such that, for a.e. \(s\in E_t^\ell\),
\[
\alpha_a^{(\ell)}(s)
+
\sum_{j\in M_{k_\ell}(a)}
\beta_{a,j}^{(\ell)}(s)
=1.
\]
Moreover, the corresponding estimate in Lemma~\ref{lem3.1} holds with error \(\gamma/(4|K|)\).

For \(s\in E_t^\ell\), write
\[
\mathcal A_t(s)=\mathcal A_{t,\ell},
\qquad
F_a(s)=F_{a,\ell}(s),
\]
and set
\[
\alpha_a(s)=\alpha_a^{(\ell)}(s),
\qquad
\beta_{a,j}(s)=\beta_{a,j}^{(\ell)}(s),\quad I_a(s)=I_{a,\ell}.
\]
Since \(E_t^0\cap E_t^1=\varnothing\), these functions are well defined
and measurable on \(E_t\).

Define
\[
r_a(s)=F_a(s)\alpha_a(s),
\qquad
b_{a,j}(s)=F_a(s)\beta_{a,j}(s).
\]
Then, for a.e. \(s\in E_t\),
\[
r_a(s)+\sum_{j\in M_n(a)}b_{a,j}(s)=F_a(s),
\qquad a\in\mathcal A_t(s),
\]
where \(n=n(t,s)\) and \(\sigma=\sigma(t,s)\).

Since
\(
|\mathcal A_{t,\ell}|\leq |K|
\),
the lemma is applied at most \(2|K|\) times. Summing the corresponding
estimates yields
\[
\begin{aligned}
	&\int_{E_t}
	\sum_{a\in\mathcal A_t(s)}
	\biggl(
	|r_a(s)|^p\rho_a(s)
	+
	\sum_{j\in M_n(a)}
	|b_{a,j}(s)|^p\rho_j(\sigma)
	\biggr)\,ds\\
	&\qquad\leq
	\int_{E_t}
	\sum_{a\in\mathcal A_t(s)}
	|F_a(s)|^p
	\biggl(
	\rho_a(s)^{-1/(p-1)}
	+
	\sum_{j\in M_n(a)}
	\rho_j(\sigma)^{-1/(p-1)}
	\biggr)^{1-p}ds
	+
	\frac{\gamma}{2}.
\end{aligned}
\]

Set
\[
F_t^0:=E_t^0+\tau,
\qquad
F_t^1:=E_t^1+\tau-1.
\]
For \(j\in I\) and \(u\in[0,1)\), define
\[
(w_t^{(1)})_j(u):=
\begin{cases}
	-b_{a,j}(u-\tau),
	&
	
		u\in F_t^0,\;\ a\in\mathcal A_{t,0},\;
		j\in M_{k_0}(a),
	\\[2mm]
	-b_{a,j}(u+1-\tau),
	&
		u\in F_t^1,\;\ a\in\mathcal A_{t,1},\;
		j\in M_{k_1}(a),
\\[2mm]
	0,
	&
	\text{otherwise}.
\end{cases}
\]
For each fixed \(\ell\), the sets \(M_{k_\ell}(a)\),
\(a\in\mathcal A_{t,\ell}\), are pairwise disjoint, while
\(F_t^0\cap F_t^1=\varnothing\). Hence the definition is unambiguous and
\(w_t^{(1)}\) is measurable.

For \(s\in E_t\) and \(a\in\mathcal A_t(s)\), put
\(
q_a(s)=|I_a(s)|.
\)
Then
\[
q_a(s)\leq |K|,
\qquad
|F_a(s)|\leq q_a(s)M_{f,g}.
\]
Moreover, for every \(i\in I_a(s)\),
\[
\Psi_i(t,s)
=
\biggl(
\rho_a(s)^{-1/(p-1)}
+
\sum_{j\in M_n(a)}
\rho_j(\sigma)^{-1/(p-1)}
\biggr)^{1-p}.
\]
Consequently,
\[
\begin{aligned}
	&|F_a(s)|^p
	\biggl(
	\rho_a(s)^{-1/(p-1)}
	+
	\sum_{j\in M_n(a)}
	\rho_j(\sigma)^{-1/(p-1)}
	\biggr)^{1-p}\\
	&\qquad\leq
	q_a(s)^pM_{f,g}^p
	\biggl(
	\rho_a(s)^{-1/(p-1)}
	+
	\sum_{j\in M_n(a)}
	\rho_j(\sigma)^{-1/(p-1)}
	\biggr)^{1-p}\\
	&\qquad\leq
	|K|^{p-1}M_{f,g}^p
	\sum_{i\in I_a(s)}\Psi_i(t,s).
\end{aligned}
\]
Therefore,
we obtain
\begin{equation}\label{3.10}
\begin{aligned}
	\|w_t^{(1)}\|_{p,\rho}^p
	+
	\int_{E_t}
	\sum_{a\in\mathcal A_t(s)}
	|r_a(s)|^p\rho_a(s)\,ds
	\leq
	|K|^{p-1}M_{f,g}^p
	\int_{E_t}
	\sum_{i\in K}\Psi_i(t,s)\,ds
	+
	\frac{\gamma}{2}.
\end{aligned}
\end{equation}

\medskip
\noindent\textit{Step 4. Construction of the target term on \(E_t\).} We next construct \(w_t^{(2)}\), which produces \(g\) on \(E_t\). By the same
splitting argument and measurable selection as in the proof of Theorem \ref{thm3.1},
applied to \(\Phi_i\) on \(E_t^0=E_t\cap[0,1-\tau)\) and
\(E_t^1=E_t\cap[1-\tau,1)\), we may choose nonnegative measurable functions
\(v_{i,j}^{(2)}\), \(i\in\operatorname{supp}g\), satisfying
\[
\sum_{j\in M_{n(t,s)}(i)}v_{i,j}^{(2)}(s)=1,
\qquad s\in E_t,
\]
and, after combining the estimates on \(E_t^0\) and \(E_t^1\),
\[
\begin{aligned}
	\int_{E_t}
	\sum_{j\in M_{n(t,s)}(i)}
	|v_{i,j}^{(2)}(s)|^p
	\rho_j(\sigma(t,s))\,ds\leq
	\int_{E_t}\Phi_i(t,s)\,ds
	+
	\frac{\gamma}{2|K|M_{f,g}^p}.
\end{aligned}
\]
For \(j\in I\) and \(u\in[0,1)\), define
\[
(w_t^{(2)})_j(u):=
\begin{cases}
	g_i(u-\tau)v_{i,j}^{(2)}(u-\tau),
	&
	u\in F_t^0, i\in\operatorname{supp}g,\;
	j\in M_{k_0}(i),
	\\[2mm]
	g_i(u+1-\tau)v_{i,j}^{(2)}(u+1-\tau),
	&
	u\in F_t^1, i\in\operatorname{supp}g,\;
	j\in M_{k_1}(i),
	\\[2mm]
	0,
	&
	\text{otherwise}.
\end{cases}
\]
For each fixed \(\ell\), the sets \(M_{k_\ell}(i)\),
\(i\in\operatorname{supp}g\), are pairwise disjoint, while
\(F_t^0\cap F_t^1=\varnothing\). Hence the definition is unambiguous and
\(w_t^{(2)}\) is measurable. Moreover,
\[
(T_tw_t^{(2)})_i(s)=g_i(s),
\qquad i\in\operatorname{supp}g,
\]
for a.e. \(s\in E_t\), and
\begin{equation}\label{3.11}
\|w_t^{(2)}\|_{p,\rho}^p
\leq
M_{f,g}^p
\int_{E_t}\sum_{i\in K}\Phi_i(t,s)\,ds
+
\frac{\gamma}{2}.
\end{equation}

\medskip
\noindent\textit{Step 5. Final estimates.} Moreover, the supports of
\(w_t^{(0)}\), \(w_t^{(1)}\), and \(w_t^{(2)}\) are pairwise disjoint up to null
sets. The support of \(w_t^{(0)}\) is disjoint from those of
\(w_t^{(1)}\) and \(w_t^{(2)}\). Indeed, the map
\(
s\longmapsto \sigma(t,s)
\)
is a bijection of \([0,1)\) onto itself, and
\[
B_t=\sigma(t,E_t^c),
\qquad
F_t^0\cup F_t^1=\sigma(t,E_t).
\]
Hence
\[
B_t\cap(F_t^0\cup F_t^1)=\varnothing.
\]
Since \(w_t^{(0)}\) is supported in \(K\times B_t\), whereas
\(w_t^{(1)}\) and \(w_t^{(2)}\) are supported in
\(I\times(F_t^0\cup F_t^1)\), the asserted disjointness follows.
 Moreover, the supports of \(w_t^{(1)}\) and \(w_t^{(2)}\) are disjoint
 up to null sets. Indeed, if
 \(j\in M_{k_\ell}(a)\cap M_{k_\ell}(i)\), then the uniqueness of the
 \(k_\ell\)-th ancestor gives \(a=i\), contradicting
 \(a\notin K\) and \(i\in\operatorname{supp}g\subset K\).
 
 Set 
\[
w_t=w_t^{(0)}+w_t^{(1)}+w_t^{(2)}.
\]
Combining the estimates in \eqref{3.9}--\eqref{3.11} and applying \eqref{3.8}, we obtain
\[
\begin{aligned}
	\|w_t\|_{p,\rho}^p
	&=
	\|w_t^{(0)}\|_{p,\rho}^p
	+
	\|w_t^{(1)}\|_{p,\rho}^p
	+
	\|w_t^{(2)}\|_{p,\rho}^p \\
	&\leq
	\frac{\eta^p}{4}
	+
	|K|^{p-1}M_{f,g}^p
	\int_{E_t}\sum_{i\in K}
	\bigl(\Phi_i(t,s)+\Psi_i(t,s)\bigr)\,ds
	+
	\gamma \\
	&<
	\frac{\eta^p}{4}
	+
	|K|^{p-1}M_{f,g}^p\varepsilon
	+
	\gamma \\
	&<
	\frac{\eta^p}{4}+\frac{\eta^p}{8}+\frac{\eta^p}{8}
	=
	\frac{\eta^p}{2}.
\end{aligned}
\]
Hence
\[
\|w_t\|_{p,\rho}<\eta,
\qquad
f+w_t\in B(f,\eta)\subset U.
\]

It remains to estimate \(T_t(f+w_t)-g\). On \(E_t\), the construction of
\(w_t^{(2)}\) gives the target coordinates \(g_i\),
\(i\in\operatorname{supp}g\), and the construction of \(w_t^{(1)}\) leaves only
the residual terms \(r_a\), \(a\in\mathcal A_t(s)\). Hence, by  \eqref{3.8} and \eqref{3.10}
\[
\begin{aligned}
	&\sum_{i\in I}\int_{E_t}
\bigl|(T_t(f+w_t))_i(s)-g_i(s)\bigr|^p
\rho_i(s)\,ds\\
&\qquad=
\sum_{i\in I}\int_{E_t}
\bigl|
\bigl(T_t(f+w_t^{(0)}+w_t^{(1)}+w_t^{(2)})\bigr)_i(s)
-g_i(s)
\bigr|^p
\rho_i(s)\,ds\\
&\qquad=
\int_{E_t}
\sum_{a\in\mathcal A_t(s)}
|r_a(s)|^p\rho_a(s)\,ds\\
	&\qquad\leq
	|K|^{p-1}M_{f,g}^p
	\int_{E_t}\sum_{i\in K}\Psi_i(t,s)\,ds
	+\frac{\gamma}{2}\\
	&\qquad<
	|K|^{p-1}M_{f,g}^p\varepsilon+\frac{\gamma}{2}
	<
	\frac{\eta^p}{4}.
\end{aligned}
\]

On \(E_t^c\), by the definition of \(B_t\), the term \(w_t^{(0)}\) cancels
\(f\), while \(w_t^{(1)}\) and \(w_t^{(2)}\) do not contribute, since they are
supported on the image of \(E_t\) under \(s\mapsto\sigma(t,s)\). Hence
\[
\bigl(T_t(f+w_t)\bigr)_i(s)=0,
\qquad i\in I,
\]
for a.e. \(s\in E_t^c\).
Since   \(m(E_t^c)<1-\alpha\), it follows from
\eqref{3.7} that,
\[
\sum_{i\in I}\int_{E_t^c}
\bigl|(T_t(f+w_t))_i(s)-g_i(s)\bigr|^p
\rho_i(s)\,ds
=
\sum_{i\in K}\int_{E_t^c}|g_i(s)|^p\rho_i(s)\,ds
<
\frac{\eta^p}{4}.
\]
Combining the estimates on \(E_t\) and \(E_t^c\), we get
\[
\|T_t(f+w_t)-g\|_{p,\rho}^p
<
\frac{\eta^p}{2}.
\]
Hence
\[
T_t(f+w_t)\in B(g,\eta)\subset V.
\]
Therefore \(t\in N(U,V)\). Since \(t\in A\) was arbitrary, $
A\subset N(U,V)$.
Since \(A\in\mathcal F\) and \(\mathcal F\) is upward hereditary, we obtain
\(
N(U,V)\in\mathcal F.
\)
Thus \((T_t)_{t\geq0}\) is \(\mathcal F\)-transitive on \(L^p_\rho(L(G))\).
\end{proof}
For the special case \(p=1\), the corresponding quantities are defined in terms
of pointwise infima of the weights. More precisely, define
\[
\Phi_i^{(1)}(t,s)
=
\inf_{j\in M_{n(t,s)}(i)}
\rho_j(\sigma(t,s))
\]
and
\[
\Psi_i^{(1)}(t,s)
=
\inf\left\{
\rho_{K_{n(t,s)}(i)}(s),
\inf_{j\in M_{n(t,s)}(K_{n(t,s)}(i))}
\rho_j(\sigma(t,s))
\right\}.
\]

\begin{cor}
	Let \(G\) be a rootless directed tree without leaves. Let
	\(\rho=(\rho_i)_{i\in I}\) be a \(1\)-admissible weight such that the left
	translation semigroup \((T_t)_{t\geq0}\) is strongly continuous on
	\(L^1_\rho(L(G))\).
	
	For every finite nonempty set \(K\subset I\), \(\varepsilon>0\), and
	\(0<\alpha<1\), define
	\[
	\mathcal N_1(K,\varepsilon,\alpha)
	=
	\Biggl\{
	t\geq0:
	\inf_{\substack{E\subset[0,1)\ \mathrm{measurable}\\ m(E)>\alpha}}
	\int_E
	\sum_{i\in K}\bigl(\Phi_i^{(1)}(t,s)+\Psi_i^{(1)}(t,s)\bigr)\,ds
	<\varepsilon
	\Biggr\}.
	\]
	Then the following assertions are equivalent:
	\begin{enumerate}
		\renewcommand{\labelenumi}{\textnormal{(\roman{enumi})}}
		\item \((T_t)_{t\geq0}\) is \(\mathcal F\)-transitive on 	\(L^1_\rho(L(G))\);
		\item for every finite nonempty set \(K\subset I\), every
		\(\varepsilon>0\), and every \(0<\alpha<1\),
		\[
		\mathcal N_1(K,\varepsilon,\alpha)\in\mathcal F .
		\]
		\item \((T_t)_{t\geq0}\) is topologically \(\mathcal F\)-recurrent on 	\(L^1_\rho(L(G))\).
	\end{enumerate}
\end{cor}

\begin{proof}
	The proof is the same as that of the preceding theorem. For
	\(\mathrm{(iii)}\Rightarrow\mathrm{(ii)}\), replace the weighted
	H\"older estimates by
	\[
	\left(\inf_{j\in J}\omega_j\right)
	\biggr|\sum_{j\in J}z_j\biggr|
	\leq
	\sum_{j\in J}|z_j|\omega_j
	\]
	and
	\[
	\min\left\{\omega_0,\inf_{j\in J}\omega_j\right\}
	\biggr|z_0-\sum_{j\in J}z_j\biggr|
	\leq
	|z_0|\omega_0+\sum_{j\in J}|z_j|\omega_j.
	\]
	For \(\mathrm{(ii)}\Rightarrow\mathrm{(i)}\), use the case \(p=1\) of
	Lemma~\ref{lem3.1}. The remaining argument is unchanged.
\end{proof}
\section{Some examples and conclusions}
In this section, we illustrate the preceding criteria by means of two
classes of directed metric trees. The first example concerns rooted
homogeneous trees and shows how the branching degree determines the
dynamical behavior of the corresponding translation semigroup. The second
deals with a rootless tree having a free left end and highlights the
essential role of the ancestor term in the rootless criterion.

\begin{ex}\label{ex:homogeneous-tree}
	Let \(1<p<\infty\), and let
	\[
	\mathcal F_{\mathrm{tail}}
	=
	\{A\subset\mathbb R_+:[R,\infty)\subset A
	\text{ for some }R>0\}.
	\]
	Then \(\mathcal F_{\mathrm{tail}}\) is a proper finite-invariant Furstenberg
	family.
	
For \(d\in\mathbb N\), let \(G_d\) be the rooted \(d\)-homogeneous directed
tree, that is, every vertex has exactly \(d\) children, and consider the
constant weight \(\rho_i(s)=1\). This weight is \(p\)-admissible. By
Lemma~\ref{lem2.2}, the corresponding left translation semigroup is strongly
continuous on \(L^p(L(G_d))\).
	
	If \(d\geq2\), then
	\[
	\lvert M_n(i)\rvert=d^n,
	\qquad
	\Phi_i(t,s)=d^{(1-p)n(t,s)}
	\leq d^{(1-p)\lfloor t\rfloor}.
	\]
	Hence, for every finite nonempty \(K\subset I\), as \(t\to\infty\)
	\[
	\int_0^1\sum_{i\in K}\Phi_i(t,s)\,ds
	\leq |K|d^{(1-p)\lfloor t\rfloor}\longrightarrow0.
	\]
 Therefore, for every \(\varepsilon>0\) and
	\(0<\alpha<1\), there exists \(R>0\) such that
	\[
	(R,\infty)\subset\mathcal N(K,\varepsilon,\alpha).
	\]
	Thus \(\mathcal N(K,\varepsilon,\alpha)\in\mathcal F_{\mathrm{tail}}\), and
	Theorem~\ref{thm3.1} yields that the semigroup is
	\(\mathcal F_{\mathrm{tail}}\)-transitive and topologically
	\(\mathcal F_{\mathrm{tail}}\)-recurrent.
	
	If \(d=1\), then \(\operatorname{card}M_n(i)=1\) and hence
	\(\Phi_i(t,s)=1\). For \(K=\{i\}\),
	\[
	\inf_{\substack{E\subset[0,1)\\m(E)>\alpha}}
	\int_E\Phi_i(t,s)\,ds
	=
	\inf_{\substack{E\subset[0,1)\\m(E)>\alpha}}m(E)
	=
	\alpha.
	\]
	Thus, whenever \(0<\varepsilon<\alpha\),
	\(
	\mathcal N(\{i\},\varepsilon,\alpha)=\varnothing.
	\)
	Since \(\mathcal F_{\mathrm{tail}}\) is proper, the semigroup is neither
	\(\mathcal F_{\mathrm{tail}}\)-transitive nor topologically
	\(\mathcal F_{\mathrm{tail}}\)-recurrent. 
\end{ex}

\begin{ex}
	\label{ex:free-left-end}
Let \(1<p<\infty\), \(d\geq2\). Put \(I_d=\{1,\ldots,d\}\), and let
\(I_d^n\) denote the set of all multi-indices
\(\mathbf i=(i_1,\ldots,i_n)\) with entries in \(I_d\).
Define
\[
V
=
\{v_k:k\leq0\}
\cup
\{w_{\mathbf i}:\mathbf i\in I_d^n,\ n\geq1\}
\]
by
\[
\operatorname{Chi}(v_k)=\{v_{k+1}\}\quad(k\leq-1),
\qquad
\operatorname{Chi}(v_0)=\{w_i:i\in I_d\},
\]
and
\[
\operatorname{Chi}(w_{\mathbf i})
=
\{w_{\mathbf i,r}:r\in I_d\},
\qquad
\mathbf i\in I_d^n.
\]
Thus \(G\) is a rootless directed tree with a free left end.

Define \(\ell:E\to\mathbb Z\) by
\[
\ell\bigl((v_{k-1},v_k)\bigr)=k
\quad(k\leq0),
\qquad
\ell\bigl((v_0,w_i)\bigr)=1
\quad(i\in I_d)
\]
and
\[
\ell\bigl((w_{\mathbf i},w_{(\mathbf i,r)})\bigr)=n+1,
\qquad
\mathbf i\in I_d^n,\quad r\in I_d,\quad n\geq1.
\]
	Fix
	\(
	1<q<d^{p-1}
	\)
	and define
	\[
	\rho_i(s)=q^{\ell(e_i)},
	\qquad i\in I,\quad s\in[0,1).
	\]
	Let \(m=\ell(e_i)\), \(n=n(t,s)\), and
	\(\sigma=\sigma(t,s)\). Since
	\[
	|M_n(i)|\leq d^n,
	\qquad
	\rho_j(\sigma)=q^{m+n}
	\quad (j\in M_n(i)),
	\]
	we have
	\[
	\begin{aligned}
		\biggl(
		\sum_{j\in M_n(i)}
		\rho_j(\sigma)^{-1/(p-1)}
		\biggr)^{p-1}
		\leq
		\frac{1}{\rho_i(s)}
		\biggr(\frac{d^{p-1}}{q}\biggr)^n  \leq
		\frac{C e^{(\log C)t}}{\rho_i(s)},
	\end{aligned}
	\]
	where \(C=d^{p-1}/q>1\). Hence \(\rho\) is \(p\)-admissible.
	
	For \(n>|m|\), one has
	\[
	|M_n(i)|=d^{\,n+\min\{m,0\}},
	\qquad
	|M_n(K_n(i))|=d^{\max\{m,0\}}.
	\]
	Therefore, with
	\(
	\theta=\frac{q}{d^{p-1}}\in(0,1),
	\)
	we obtain
	\[
	\Phi_i(t,s)
	=
	q^m d^{(1-p)\min\{m,0\}}\theta^n
	\]
	and
	\[
	\begin{aligned}
		\Psi_i(t,s)
		=
		q^m
		\bigl(
		q^{n/(p-1)}+d^{\max\{m,0\}}
		\bigr)^{1-p} \leq q^{m-n}.
	\end{aligned}
	\]
	Since \(n(t,s)\geq\lfloor t\rfloor\), for every finite nonempty
	\(K\subset I\) there exist \(A_K,B_K,R_K>0\) such that
	\[
	\sum_{i\in K}
	\bigl(\Phi_i(t,s)+\Psi_i(t,s)\bigr)
	\leq
	A_K\theta^{\lfloor t\rfloor}
	+
	B_Kq^{-\lfloor t\rfloor}
	\]
	for all \(t\geq R_K\) and \(s\in[0,1)\). Consequently,
	\[
	\int_0^1
	\sum_{i\in K}
	\bigl(\Phi_i(t,s)+\Psi_i(t,s)\bigr)\,ds
	\longrightarrow0
	\]
	as \(t\to\infty\). Taking \(E=[0,1)\), Theorem~\ref{thm3.5}
	implies that \((T_t)_{t\geq0}\) is
	\(\mathcal F_{\mathrm{tail}}\)-transitive and
	topologically \(\mathcal F_{\mathrm{tail}}\)-recurrent.
	
	The ancestor term cannot be omitted. Indeed, for the constant weight
	\(\rho_i=1\), which is also \(p\)-admissible, one has
	\[
	\Phi_i(t,s)\longrightarrow0,
	\qquad
	\Psi_i(t,s)
	=
	\bigl(1+d^{\max\{m,0\}}\bigr)^{1-p}
	\]
	for all sufficiently large \(t\), where \(m=\ell(e_i)\). Hence, if
	\(m(E)>\alpha\),
	\[
	\int_E
	\bigl(\Phi_i(t,s)+\Psi_i(t,s)\bigr)\,ds
	>
	\alpha
	\bigl(1+d^{\max\{m,0\}}\bigr)^{1-p}.
	\]
	Thus, for sufficiently small \(\varepsilon>0\),
	\(\mathcal N(\{i\},\varepsilon,\alpha)\notin
	\mathcal F_{\mathrm{tail}}\). This shows that forward branching alone
	does not guarantee mixing on a rootless tree with a free left end.
\end{ex}
% ------------------------------------------------------------------------

 \vspace{1em} % 调整间距
\noindent \textbf{Conflict of Interest.} The authors declare no competing interests with the journal's editors, reviewers, or readers.

\noindent \textbf{Funding.}
This work was supported by  the National Natural Science Foundation of
China (Grant Nos. 12571088) and  the Natural Science Foundation of Henan  (No.262300421862).

\noindent\textbf{Data Availability Statement.}
No datasets were generated or analyzed in this theoretical study.

% ------------------------------------------------------------------------
\end{document}